\documentclass{article}

\usepackage{arxiv}

\usepackage[utf8]{inputenc} 
\usepackage[T1]{fontenc}    
\usepackage{hyperref}       
\usepackage{url}            
\usepackage{booktabs}       
\usepackage{amsfonts}       
\usepackage{nicefrac}       
\usepackage{microtype}      
\usepackage{mathtools}
\usepackage{amsthm}
\usepackage[dvipsnames]{xcolor}

\DeclarePairedDelimiter\abs{\lvert}{\rvert}%
\DeclarePairedDelimiter\norm{\lVert}{\rVert}%

\makeatletter
\let\oldabs\abs
\def\abs{\@ifstar{\oldabs}{\oldabs*}}
\let\oldnorm\norm
\def\norm{\@ifstar{\oldnorm}{\oldnorm*}}
\makeatother

\title{A geometrical summation method for the Riemann zêta function}

\author{
  Ulysse REGLADE \\
  Student at Mines Paristech\\
  \texttt{ulysse.reglade@mines-paristech.fr} \\
}

\newtheorem{theorem}{Theorem}[section]

\newtheorem{definition}{Definition}[section]

\begin{document}
\maketitle

\begin{abstract}
\textbf{Disclaimer}: \textit{I am no professional mathematician. Though, it seems to me that the intuition presented in this preprint has not yet been explored by the mathematical community. This intuition seems very consistent and makes predictions, these predictions can be verified numerically. Because this article is a preliminary work, some parts of the proof are not complete. They will be explicitly indicated}.
\newline

In this paper, we introduce a geometrical summation method that makes the original Riemann series converge over the critical strip. This method gives an analytical function, that coincides with zêta. This point of view allows us to introduce a quantity of
interest that seems to give a characterization of the non-trivial zeros of the
Riemann z\^eta function.

For: $z=x+iy, x\in\mathbb{R}_+^*, y\in\mathbb{R}^*$, zêta can be defined as:
\begin{center}
$\boxed{\zeta(z) = \underset{N\rightarrow\infty}{\lim} \sum_{n=1}^{N}\frac{1}{n^z} + \frac{1}{(N+1)^z}\frac{1}{1-z}\left( 1-x-\frac{y}{\tan(y\ln(\frac{N+2}{N+1}))} \right)}$
\end{center}

From here, we can show that if $z$ is a non trivial zero of zêta, the following quantity converges to a relative integer seemingly even. Though the reciprocal is false, it is natural to compute it for the known non-trivial zeros of zêta:

\begin{center}
$\boxed{
\zeta(z)=0 \Rightarrow \underset{n\rightarrow\infty}{\lim} \frac{-y\ln(n+1) - \underset{sum}{\arg}(\zeta_n(z)) + \arctan(\frac{y}{1-x})}{\pi}= \mathcal{U}_z, \mathcal{U}_z\in\mathbb{Z}
}$
\end{center}

\begin{center}
$\underset{sum}{\arg}(\zeta_n(z)) = \sum_{m=1}^{n}\arg\left(\frac{\sum_{l=1}^{m+1}\frac{1}{l^z}}{\sum_{l=1}^{m}\frac{1}{l^z}}\right)$
\end{center}

At the end of this document you can find the table of these values computed for the first 30 know zeros of zêta. This last identity seems to be deeply correlated with the Riemann hypothesis.

In this paper, we are taking full advantage of formal calculation algorithms, especially to compute some asymptotic expansions. Here is the GitHub link to all codes used for this work:

\begin{center}
\href{https://github.com/UlysseREGLADE/Zeta}{https://github.com/UlysseREGLADE/Zeta}
\end{center}

\end{abstract}


\bigskip
\bigskip
\bigskip
\bigskip
\bigskip
\bigskip
\bigskip
\bigskip
\bigskip
\bigskip


\section{Introduction}

For convenience, we define for $ z \in \Omega $:

\begin{equation}
\begin{split}
&\zeta_n (z) = \sum_{j=1}^{n} \frac{1}{j^z}, n \in \mathbb{N}^* \\
&\zeta_0(z) = 0
\end{split}
\end{equation}

In the whole paper, $ z $ is defined as a complex number of $ \Omega $ an open set of $\mathbb{C}$. $ x $ and $ y $ are defined with no ambiguity as its real and complex parts:

\begin{equation}
\begin{split}
&x = \Re z \\
&y = \Im z
\end{split}
\end{equation}

We are not going to make a real difference between the complex representation of $z$ and its vectorial form $(x, y)$. Therefore, we will allow ourselves to use these representations quite freely.

Let’s now discuss the intuition that gave rise to this paper. I’m uncomfortable with the usual way of defining the zêta function \cite{BOOK:1}.

\begin{equation}
\zeta(z) = \sum_{n=1}^\infty\frac{1}{n^z}, x>1
\end{equation}

In this notation, information is lost on how the partials sums converge to zêta. To me, a more meaningful, but still not perfect way to define zêta would be something like:

\begin{equation}
\zeta(z) = \overline{\left\{ \zeta_n(z) \right\}_{n \in \mathbb{N}}}
\end{equation}

This definition is not equivalent to the usual one, it is indeed defined even if $x=1$. Though it is possible to do much better. What we are going to show is that, in a way, the zêta function can be defined as the center of a logarithmic asymptotic spiral described by the partial sums of the Riemann series  $\left\{\zeta_n(z)\right\}_{n\in\mathbb{N}}$. In fact, this sequence describes a converging spiral for $x>1$, a circle for $x=1$, and a diverging spiral for $x<1$. In each case, the value of zêta is the center of these spirals. We are going to take advantage from this fact to define zeta for: $x>0$.

\begin{figure}[h]
\begin{center}
\includegraphics[scale=0.4]{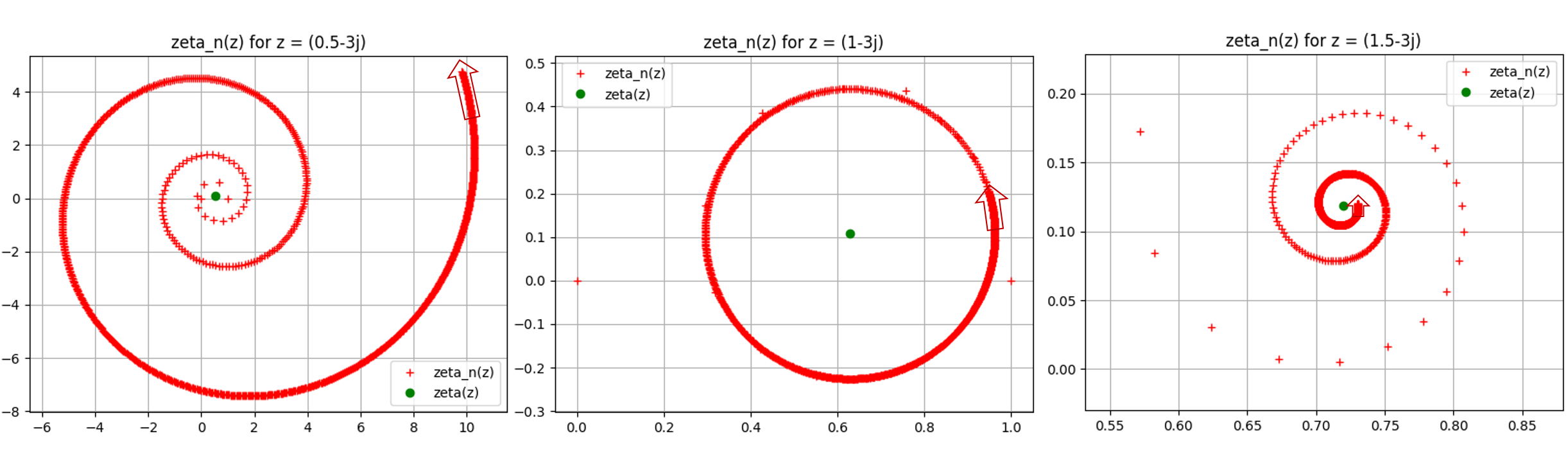}
\caption{Plot of the Riemann series, for: $z=\frac{1}{2}-3i$, $z=1-3i$, ans $z=\frac{3}{2}-3i$}
\label{Goe1}
\end{center}
\end{figure}

\section{The radial convergence, a geometrical summation method for the Riemann series}

\subsection{The radial convergence for the special case $ z \in \left\{ 1+iy, y \in \mathbb{R}^* \right\} $.}

In this subsection, $ \Omega $ is defined as:

\begin{equation}
\Omega = \left\{ 1+iy, y \in \mathbb{R}^* \right\}
\end{equation}

Let us define:

\begin{equation}
\Delta_n (z) = \left\{ \zeta_n(z) + \frac{te^{i\frac{\pi}{2}}}{(n+1)^z}, t \in \mathbb{R} \right\}, z \in \Omega, n \in \mathbb{N}
\end{equation}

$ \Delta_n(z) $ is a straight line of the complex plan. Considering $ \frac{1}{n^z} $ as a vector, it is clearly not collinear to $ \frac{1}{(n+1)^z} $. Indeed: $ \arg(\frac{n}{n+1}^z) = y \ln(\frac{n}{n+1}) \neq 0 $. Therefore, $ \Delta_n(z) $ and $ \Delta_{n+1}(z) $ must have an unique intersection point. Let us call it $ c_n(z) $:

\begin{equation}
c_n(z) = \Delta_n(z) \cap \Delta_{n+1}(z), n \in \mathbb{N}
\end{equation}

Let us adopt the following notation:

\begin{equation}
\delta_n(z) = c_{n}(z) - c_{n-1}(z), n \in \mathbb{N}^*
\end{equation}

\begin{figure}[h]
\begin{center}
\includegraphics[scale=0.5]{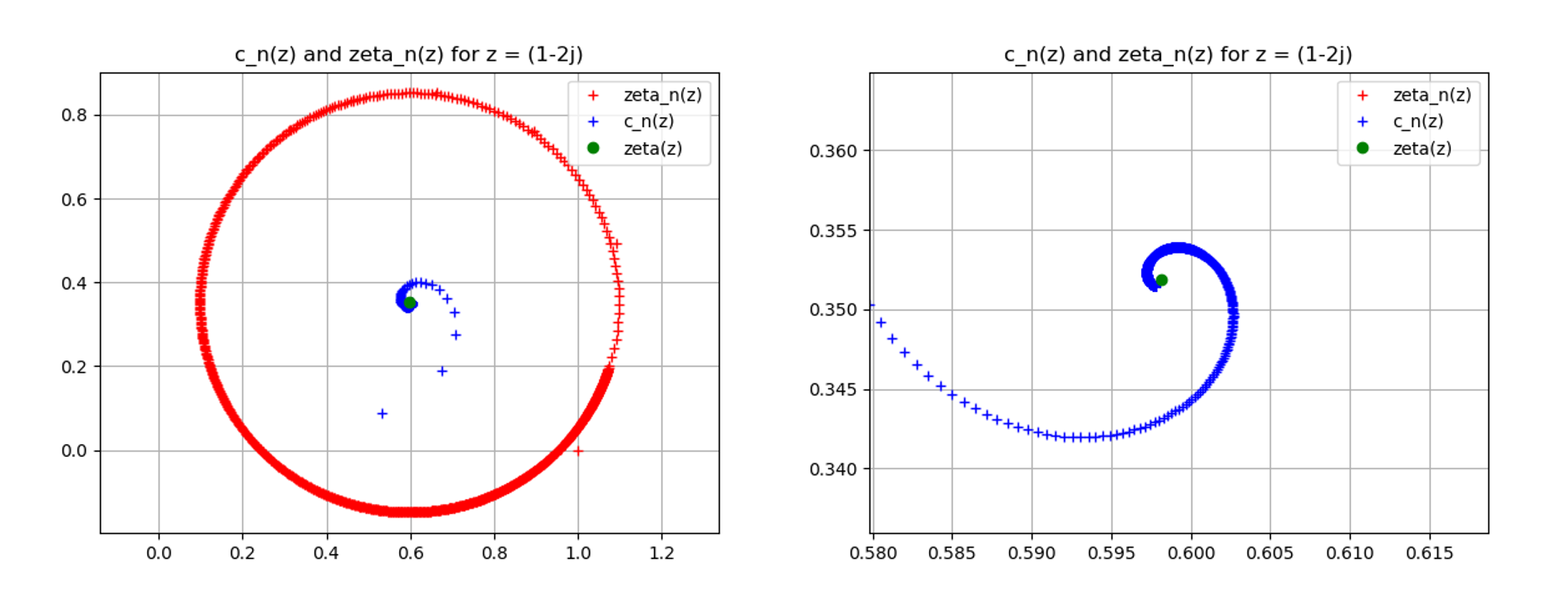}
\caption{Plot of $\left\{ c_n(z) \right\}_{n\in\mathbb{N}}$ for: $ z=1-2i $}
\label{Conv1}
\end{center}
\end{figure}

\textbf{Figure \ref{Conv1}} is a plot $\left\{ c_n(z) \right\}_{n\in\mathbb{N}}$, and we can enunciate our first theorem:

\begin{theorem}
For $ z \in \Omega $, $ \sum \delta_n(z) $ converges by Riemann sommation, and we can define with no ambiguity:
\begin{equation}
c(z) = \underset{n\rightarrow\infty}{\lim} c_n(z)
\end{equation}
\end{theorem}

\begin{figure}[h]
\begin{center}
\includegraphics[scale=1]{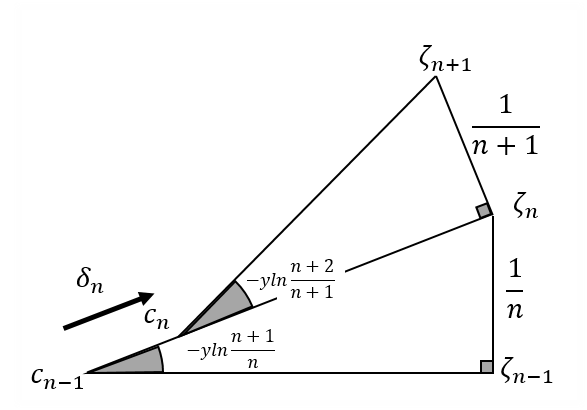}
\caption{Geometrical view of the situation}
\label{Goe1}
\end{center}
\end{figure}

\begin{proof}

Using geometrical relationships, we can write from the \textbf{Figure \ref{Goe1}}:

\begin{equation}
\abs{\delta_n(z)} = \abs{ \frac{1}{(n+1)\tan(-y \ln(\frac{n+2}{n+1}) )} - \sqrt{\frac{1}{n^2} + \frac{1}{n^2\tan( y \ln(\frac{n+1}{n}) )^2}} }
\label{delta_re_1}
\end{equation}

From \textbf{Equation \ref{delta_re_1}}, we can calculate the asymptotic expansion of $ \delta_n(z) $. This calculation is quite tedious, but at the end of the day one can show that:

\begin{equation}
\delta_n(z) \underset{n\rightarrow\infty}{=} \frac{1}{2n^2}\abs{y + \frac{1}{y}} + o(\frac{1}{n^2})
\end{equation}

The series converges by Riemann sommation. This proves the theorem.

\end{proof}

In addition, we can observe that:

\begin{equation}
\abs{c_n(z) - \zeta_n(z)}  = \frac{1}{(n+1)\tan(\abs{y}\ln(\frac{n+2}{n+1}) )} \underset{n\rightarrow\infty}{=} \frac{1}{\abs{y}} + o(1)
\label{ray_cer}
\end{equation}

Which leads us to our second theorem:

\begin{theorem}[Asymptotic circle]
Lets note $ \mathcal{C}^a_b $ the circle of radius b and center a. We have:
\begin{equation}
\overline{\left\{\zeta_n\right\}_{n\in\mathbb{N}}} \subset \mathcal{C}^{c(z)}_{\frac{1}{\abs{y}}}
\end{equation}
\label{qsd}
\end{theorem}

\begin{proof}
By triangular inequality, we can write:

\begin{equation}
\abs{\zeta_n(z)-c_n(z)} - \abs{c_n(z)-c(z)} \leq \abs{\zeta_n(z)-c(z)} \leq \abs{\zeta_n(z)-c_n(z)} + \abs{c_n(z)-c(z)}
\label{ineq}
\end{equation}

Let us denote $ d(C^{c(z)}_{\frac{1}{\abs{y}}}, \zeta_n(z)) $ the distance between the circle and the sequence:

\begin{equation}
d(\mathcal{C}^{c(z)}_{\frac{1}{\abs{y}}}, \zeta_n(z)) = \abs{ \abs{\zeta_n(z)-c(z)} - \frac{1}{\abs{y}} }
\end{equation}

We have:

\begin{equation}
\begin{split}
&\abs{\zeta_n(z)-c_n(z)} = \frac{1}{\abs{y}} + o(1) \\
&\abs{c_n(z)-c(z)} = o(1)
\end{split}
\end{equation}

From \textbf{Equation \ref{ineq}} we can conclude that:

\begin{equation}
\underset{n\rightarrow\infty}{\lim} d(\mathcal{C}^{c(z)}_{\frac{1}{\abs{y}}}, \zeta_n(z)) = 0
\end{equation}

which finishes the proof.

\end{proof}

From the \textbf{Figure \ref{Conv1}}, one can expect $ \left\{ \zeta_n(z) \right\}_{n \in \mathbb{N}} $ to populate densely the circle, which is a reasonable conjecture. It would lead to an equality and not a simple inclusion in \textbf{Theorem \ref{qsd}}. Though, the demonstration of this conjecture is not necessary for the rest of the proof.

An other reasonable hypothesis is that $c$ is actually $\zeta$. But as \textbf{Equation \ref{delta_re_1}} only applies for $ x=1 $, we need first to extend the definition of $c$ to a bigger open set and show that it is analytical and coincides with the Riemann zêta function.

\subsection{The radial convergence for $ z \in \left\{ x+iy, x>0, y \in \mathbb{R}^* \right\} $}
\label{sec:headings}

For the rest this section, we define $ \Omega $ and $\omega$ as:

\begin{equation}
\begin{split}
\Omega = \left\{ x+iy, x>0, y \in \mathbb{R}^*\right\} \\
\omega = \left\{ x+iy, x>1, y \in \mathbb{R}^*\right\}
\end{split}
\end{equation}

To understand the rest of the proof, we first need to make the following observation:

\begin{figure}[h]
\begin{center}
\includegraphics[scale=0.5]{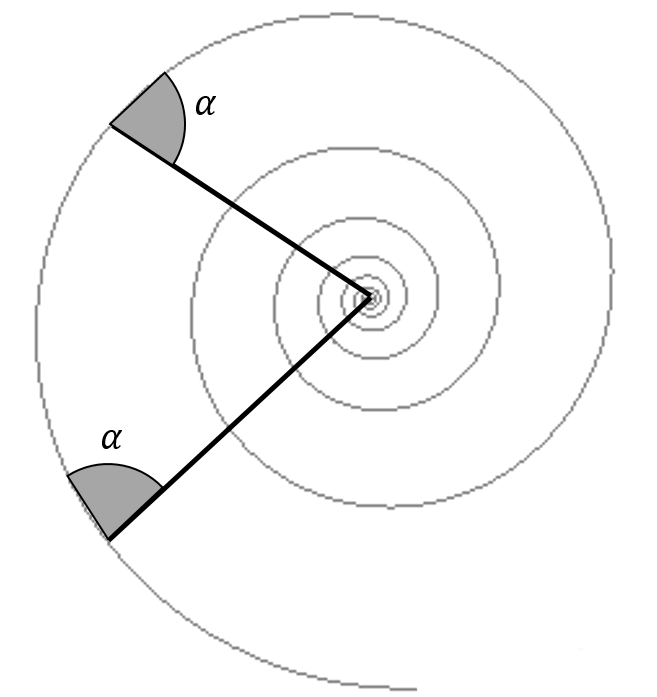}
\caption{Representation of a logarithmic spirale}
\label{Spi}
\end{center}
\end{figure}

Let us consider the integral test for convergence of $ \zeta_n(z) $:

\begin{equation}
\begin{split}
\int n^{-z} dn &= \frac{n^{1-z}}{1-z} = \frac{1}{1-z}e^{(1-x)\ln(n)}e^{-iy \ln(n)} = \frac{1}{1-z}e^{\frac{x-1}{y}\theta}e^{i\theta} \\
\theta &= -y \ln(n)
\end{split}
\end{equation}

We recognize the definition of a logarithmic spiral \cite{WEBSITE:1} in the complex plan. Such a spiral intersects its radius with a constant angle $ \alpha $, like it is shown in \textbf{Firgure \ref{Spi}}. In this case, we can compute this value, and we find:

\begin{equation}
\alpha(z) = \left\{
    \begin{array}{ll}
        \frac{\pi}{2} - \arctan(\frac{1-x}{y}) & y > 0 \\
        \frac{3\pi}{2} - \arctan(\frac{1-x}{y}) & y < 0
    \end{array}
\right.
\end{equation}

We can see that $ x = 1 $ implies $ \alpha=\frac{\pi}{2} [\pi] $. Therefore, it is natural to give a new definition to $ \Delta_n(z) $:

\begin{equation}
\Delta_n (z) = \left\{ \zeta_n(z) + \frac{te^{i\alpha}}{(n+1)^z}, t \in \mathbb{R} \right\}, z \in \Omega, n \in \mathbb{N}
\end{equation}

For the same reason than in the first section, we can define with no ambiguity the sequences $ \left\{ c_n(z) \right\}_{n \in \mathbb{N}} $ and $ \left\{ \delta_n(z) \right\}_{n \in \mathbb{N}^*} $.

\begin{figure}[h]
\begin{center}
\includegraphics[scale=0.5]{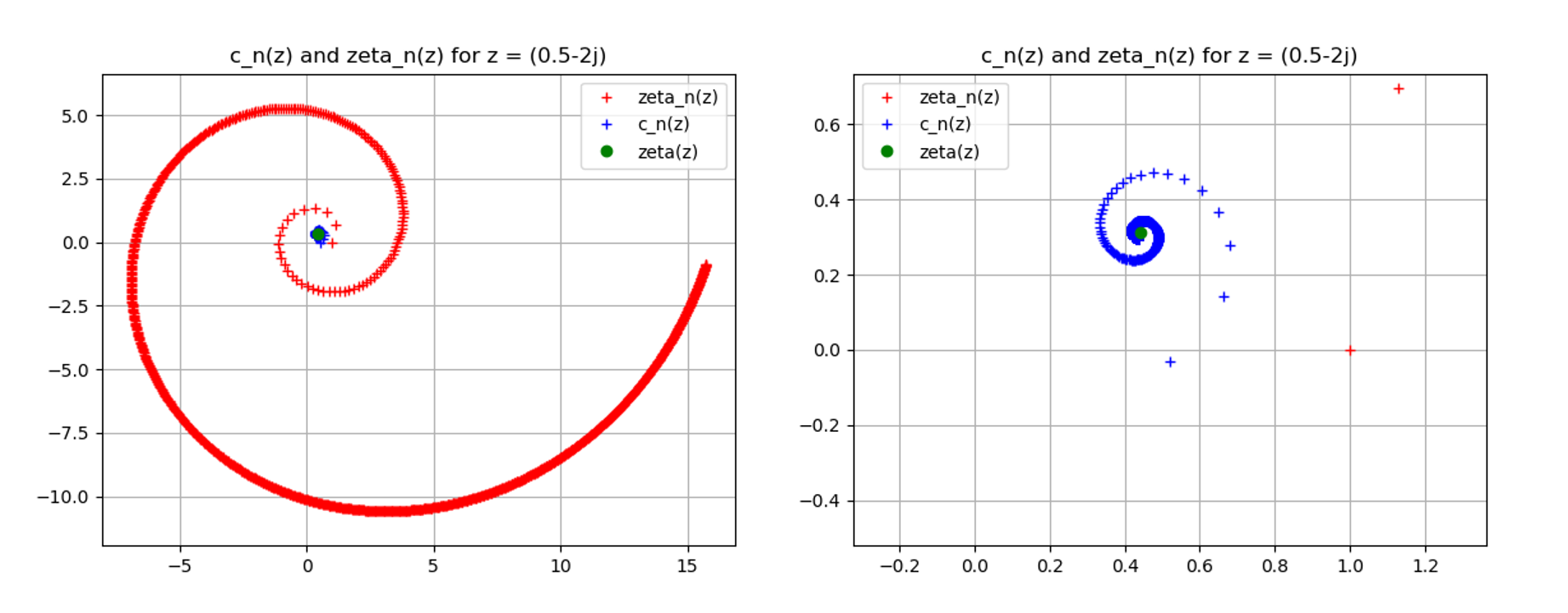}
\caption{Plot of the convergence for $ z=0.5-2i $}
\label{Conv2}
\end{center}
\end{figure}

From the \textbf{Figure \ref{Conv2}}, it seems reasonable to expect $ \left\{ c_n(z) \right\}_{n \in \mathbb{N}} $ to converge over the critical strip. It can even already conjecture that their limit coincides with $\zeta(z)$.

\subsection{Existence and continuity of $ c: \Omega \rightarrow \mathbb{C} $}

First, let us remind the definition of a domination function:

\begin{definition}
$ D_d:\mathbb{R} \rightarrow \mathbb{R} $ dominates $ f_n:\mathbb{C} \rightarrow \mathbb{C} $ over $ \Omega_d $ if and only if:

\begin{equation}
\exists n_0 \in \mathbb{N} \vert \forall z \in \Omega_d, \forall n>n_0, D_d(n) > \abs{f_n(z)}
\end{equation}
\end{definition}

We are now ready to formulate the weak version of what will be called from now the radial convergence for the Riemann zêta function.

\begin{theorem}
For $ z \in \Omega $, the sequence $ \left\{c_n(z)\right\}_{n \in \mathbb{N}} $ converges uniformly to $ c(z) $. $ c: \Omega \rightarrow \mathbb{C} $ is continuous over $ \Omega$, and we have:
\begin{equation}
c(z)=\zeta(z), z \in \omega
\end{equation}
$ c $ is complex-differentiable (i.e. holomorphic) over $ \omega \subset \Omega $, and it coincides with $ \zeta $ over this open set.
\end{theorem}

\begin{figure}[h]
\begin{center}
\includegraphics[scale=1]{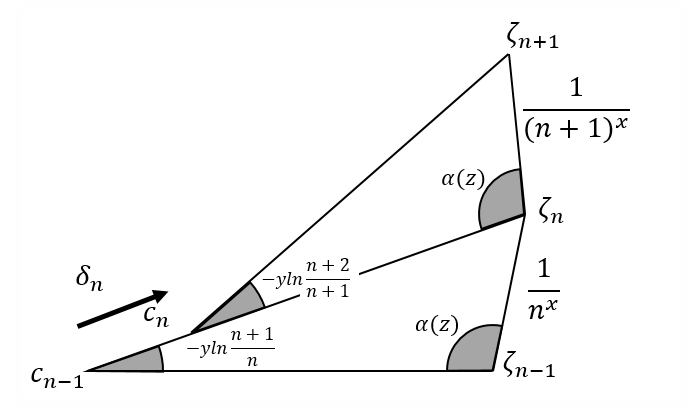}
\caption{Geometrical view of the situation for $ \alpha \in ]0,\pi[\cup]\pi,2\pi[ $}
\label{Goe2}
\end{center}
\end{figure}

\begin{proof}
First, we observe from \textbf{Figure \ref{Goe2}} that we can write $ \delta_n(z) $ as:

\begin{equation}
\delta_n(z) = \frac{1}{n^{x}\sqrt{1+(\frac{1-x}{y})^2}}\left(
\frac{1}{\sin(-y\ln\frac{n+1}{n})} - \frac{1}{(1+\frac{1}{n})^{x}}\left(\frac{1}{\tan(-y \ln(\frac{n+2}{n+1}))} + \frac{1-x}{y}\right)\right)e^{i(\alpha(z)-y \ln(n+1))}
\label{delta_n_1}
\end{equation}

To obtain this result, we have first computed the algebrical module of $ \delta_n(z) $ only using geometrical relations in the \textbf{Figure \ref{Goe2}}. Then, we compute the algebrical argument of $ \delta_n(z) $ by induction and show that it is in fact: $ \alpha(z) - y \ln(n+1) $.

We also work out the expression of $ c_0(z) $ by hand. It can be neatly written as:

\begin{equation}
c_0(z) = 1 - \frac{1}{1-z}\frac{y2^{-iy}}{\sin(y\ln2)}
\label{eq_c_0}
\end{equation}

This enable us to write:

\begin{equation}
c_n(z) = c_0(z) + \sum_{i=1}^{n} \delta_i(z)
\end{equation}

Then, just like in the first section, we compute the asymptotic expansion of $ \abs{\delta_n(z)} $ as $ n $ goes to infinity:

\begin{equation}
\abs{\delta_n(z)} = \frac{1}{n^{x}\sqrt{1+(\frac{1-x}{y})^2}}\abs{
\frac{1}{\sin(y\ln\frac{n+1}{n})} - \frac{1}{(1+\frac{1}{n})^{x}}\left(\frac{1}{\tan(y \ln(\frac{n+2}{n+1}))} + \frac{x-1}{y}\right)}
\label{dfg}
\end{equation}

We find:

\begin{equation}
\abs{\delta_n(z)} \underset{n\rightarrow\infty}{=} \frac{1}{n^{x +1}}\frac{x^2+y^2}{2\sqrt{(1-x)^2+y^2}} = \frac{1}{n^{x +1}}\frac{\abs{z}^2}{2\abs{1-z}} + o(\frac{1}{n^{1+x}})
\label{abs_delta}
\end{equation}

$ c_0(z) $ and $ \delta_n(z) $ are continuous with respect to $ z $ over $ \Omega $. It can be even say right away that they are real-differentiable.

Now, a domination function needs to be chosen. Though, it is not that easy in this case. Indeed, for large values of $y$, because of the presence of $\frac{1}{\sin(y\ln\frac{n+1}{n})}$ and $\frac{1}{\tan(y\ln\frac{n+2}{n+1})}$ in the expression of $\abs{\delta_n}$, we can expect $\abs{\delta_n}$ to explode for the first values of $n$.

From \textbf{Equation \ref{dfg}}, we can give the following expression for $\abs{\delta_n}$:

\begin{equation}
\abs{\delta_n(z)} = \frac{1}{n^{x}\abs{1-z}}\abs{
\frac{y}{\sin(y\ln\frac{n+1}{n})} - \frac{y}{(1+\frac{1}{n})^{x}\tan(y \ln(\frac{n+2}{n+1}))} + \frac{1-x}{(1+\frac{1}{n})^{x}}}
\end{equation}

We introduce: $ d_{x, y}:\mathbb{R}_+^* \rightarrow \mathbb{R} $, the function defined by:

\begin{equation}
d_{x, y}(n) = n\left(
\frac{y}{\sin(y\ln\frac{n+1}{n})} - \frac{y}{(1+\frac{1}{n})^{x}\tan(y \ln(\frac{n+2}{n+1}))} + \frac{1-x}{(1+\frac{1}{n})^{x}}\right)
\end{equation}

The limit of $d_{x, y}$ is:
\begin{equation}
\underset{n\rightarrow\infty}{\lim} d_{x, y}(n) = \frac{x^2+y^2}{2}
\end{equation}

The detailed study of $d_{x, y}$ is necessary in order to be able to properly dominate $\delta_n$.

Let us denote $\Omega_d$ a bounded open set of $\Omega$:

\begin{equation}
\begin{split}
a &= \underset{z \in \Omega_d}{inf} x \\ A &= \underset{z \in \Omega_d}{sup} x \\
b &= \underset{z \in \Omega_d}{inf} \abs{y} \\ B &= \underset{z \in \Omega_d}{sup} \abs{y}
\end{split}
\end{equation}

\begin{theorem}
Let $\mathcal{O}_b(\Omega)$ the set of the bounded open sets of $\Omega$. We can define:

\begin{equation}
n_0 : \left|\vtop{\hbox{ $\mathcal{O}_b(\Omega) \to \mathbb{N}$}\hbox{ $\Omega_d \mapsto n_0(\Omega_d)$}}\right.
\end{equation}

This function is such as:

\begin{equation}
\forall \Omega_d \in \mathcal{O}_b(\Omega), \forall z \in \Omega_d, n>n_0(\Omega_d) \Rightarrow \frac{1}{4}(x^2+y^2) < d_{x,y}(n) < \frac{3}{4}(x^2+y^2)
\end{equation}

\end{theorem}

\begin{proof}
This proof is very technical, and is not detailed in this paper. But here are the mains ideas:

\begin{itemize}

\item First, we observe that the sign of $d_{x, y}(n)$ does not depend on the sign of $y$. We can limit our study to $y>0$.

\item Then, we observe that for $n > n_B = \frac{2-e^{\frac{\pi}{B}}}{e^{\frac{\pi}{B}}-1}$, the signs of $\sin(y\ln(\frac{n+1}{n}))$ and $\tan(y\ln(\frac{n+2}{n+1}))$ do not change anymore.

\item We take advantage from the fact that the series expansion of $\sin(x)$, $cos(x)$, $\ln(x)$, and $\frac{1}{(1+x)^{\alpha>0}}$, when truncated, gives majorations or minorations of these function for $x>0$.

\item From here, we define two multivariate rational fractions that give upper and lower bounds for $d_{x, y}(n)$ starting from $n_B$, \textbf{Figure \ref{Framing}}.

\item The dominant coefficient of the denominator of these two fractions is an integer.

\end{itemize}

\begin{figure}[h]
\begin{center}
\includegraphics[scale=0.35]{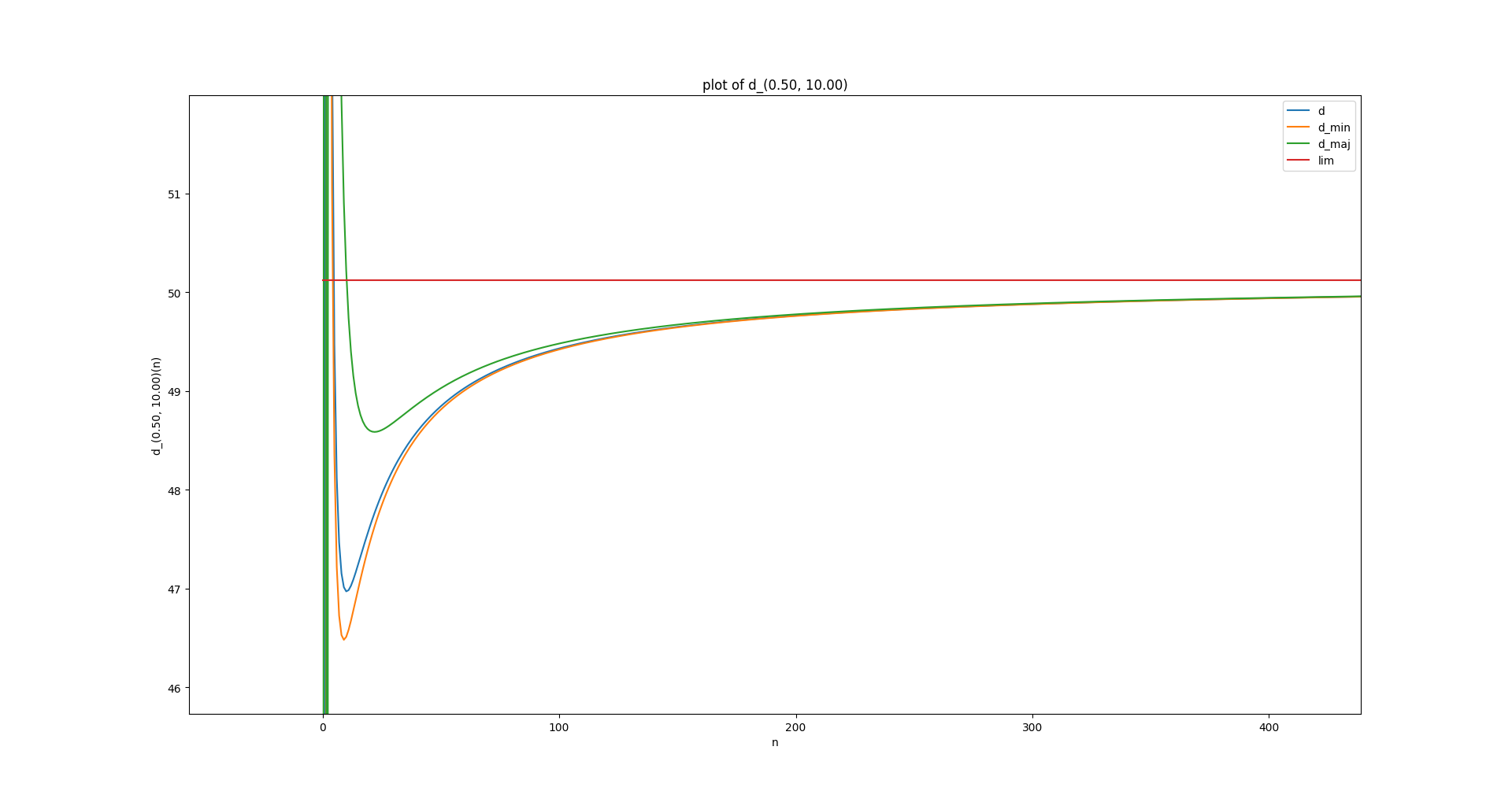}
\caption{Graph of the framing for $d_{x,y}(n)$}
\label{Framing}
\end{center}
\end{figure}

Using this last argument, it is possible to show that from a certain value of $n$, $d_{x, y}(n)$ must lie in a certain range around its limit. The code used to produce these upper and lower bounds of $d_{x, y}$ is accessible via the GitHub link at the beginning of the document.

\end{proof}

From here, it will be assumed such a domination function exists:

\begin{equation}
\begin{split}
D_{\Omega_d} &: \mathbb{R}_+^* \rightarrow \mathbb{R} \\
D_d(n) &= \frac{1}{n^{1+a}}\frac{3(A^2+B^2)}{4\abs{a-1}} \\
\end{split}
\end{equation}

This function is integrable by Riemann sommation over any $\Omega_d$ in $\Omega$. We can even specify from when this domination is effective for $\Omega_d$. It is from $n_0(\Omega_d)$, which does not depend on the choice of $z$ in $\Omega_d$.

To summarize, we have:

\begin{itemize}
    \item $ \delta_n:\Omega_d \rightarrow \mathbb{C} $ is continuous.
    \item $ \sum \abs{\delta_n(z)} $ converges over $ \Omega_d $ by Riemann sommation.
    \item $ D_d(n):\mathbb{R}_+^* \rightarrow \mathbb{R} $ dominates $ \delta_n:\Omega_d \rightarrow \mathbb{C} $ over $ \Omega_d $, and is integrable.
\end{itemize}

Therefore, $ c_0 + \sum \delta_n $ converges normally and in consequence uniformly to $ c $ over $ \Omega_d $. $ c $ is continuous over any $ \Omega_d $, so $ c $ is continuous over $ \Omega $.

Finally, we have to prove that this expression coincides with $ \zeta $ on $ \omega \subset \Omega$. It comes from the fact that:

\begin{equation}
\zeta_n(z) \underset{n\rightarrow\infty}{\rightarrow} \zeta(z), z \in \omega
\end{equation}

From \textbf{Figure \ref{Goe2}}, we can show that:

\begin{equation}
\begin{split}
&\abs{\zeta_n(z) - c_n(z)} \underset{n\rightarrow\infty}{=} \frac{1}{\abs{1-z}n^{x -1}}  + o(\frac{1}{n^{x -1}})\\
&\frac{1}{\abs{1-z}n^{x -1}} \underset{n\rightarrow\infty}{\rightarrow} 0, z \in \omega
\end{split}
\end{equation}

Therefore:

\begin{equation}
z \in \omega \Rightarrow c(z) = \zeta(z)
\end{equation}

This implies that $ c $ is complex-differentiable over $ \omega $. We now need to show that $c$ is complex-differentiable over $\Omega$, in particular on the critical strip.

\end{proof}

\subsection{Real-differenciability of $ c: \Omega \rightarrow \mathbb{C} $}

Showing that $ c $ is complex-differentiable over the critical strip is quite challenging because $ c_0 $ and $ \delta_n $ are not holomorphic functions. Therefore, the first step is to show that the Jacobian of $ \delta_n $ is well defined and continuous (real-differentiability). Only then we can verify it satisfies the Cauchy-Riemann equation (complex-differentiability).

A way to show that $c$ is indeed real-differentiable is to show that $ \sum J_{\delta_n} $ converges normally for the infinity norm. This works, but it leads to very tedious calculations that are not described here. Again, the domination is by far the most technical aspect of this proof, though the formal calculation of $ J_{\delta_n} $ gives immediately the following result:

\begin{equation}
\begin{split}
& \frac{\partial{\delta_n}_x}{\partial x} = O(\frac{\ln(n)}{n^{1+x}}) \\
& \frac{\partial{\delta_n}_x}{\partial y} = O(\frac{\ln(n)}{n^{1+x}}) \\
& \frac{\partial{\delta_n}_y}{\partial x} = O(\frac{\ln(n)}{n^{1+x}}) \\
& \frac{\partial{\delta_n}_y}{\partial y} = O(\frac{\ln(n)}{n^{1+x}})
\end{split}
\end{equation}

Therefore:

\begin{equation}
\norm{J_{\delta_n}}_{\infty} = O(\frac{\ln(n)}{n^{1+x}})
\end{equation}

We can conclude that $ \sum J_{\delta_n} $ converges for any $ (x, y) $ in $ \Omega $. We now need to find a domination function for $ J_{\delta_n} $, for the infinity norm over $ \Omega_d $.

It’s possible to reuse the technique described in the last section to find a domination function for $J_{\delta_n}$, but I think there must be a more elegant solution to this problem. For the rest of the proof, we are just going to assume that we have the following sum to converge normally for the infinity norm:

\begin{equation}
J_c((x, y)) = J_{c_0}((x, y)) + \sum_{i=1}^\infty J_{\delta_n}((x, y)), x>0, y \in \mathbb{R}^*
\end{equation}

\subsection{Complex-differenciability of $ c: \Omega \rightarrow \mathbb{C} $, equality with $ \zeta: \Omega \rightarrow \mathbb{C} $}

First, we can rewrite \textbf{Equation \ref{delta_n_1}} as:

\begin{equation}
\delta_n(z) = y\frac{1-\overline{z}}{\abs{1-z}^2}\left( \frac{2i}{n^z}\frac{1}{\frac{n+1}{n}^{2iy}-1} -\frac{i}{(n+1)^z}\frac{\frac{n+2}{n+1}^{2iy}+1}{\frac{n+2}{n+1}^{2iy}-1} + \frac{1}{(n+1)^z}\frac{1-x}{y}\right)
\end{equation}

The series $ \sum \delta_n $ converges normally, therefore we can rearrange the terms in the sum. More precisely we observe this sum is actually telescopic. We can write that:

\begin{equation}
\sum_{n=1}^N \delta_n(z) = \sum_{n=2}^N \frac{1}{n^z} + \frac{1-\overline{z}}{\abs{1-z}^2}\left( \frac{y2^{-iy}}{\sin(y\ln2)} - \frac{y}{(N+1)^z}\frac{1}{\tan(y\ln(\frac{N+2}{N+1}))} +(1-x)\frac{1}{(N+1)^z} \right)
\end{equation}

Therefore, from \textbf{Equation \ref{eq_c_0}}, we can write:

\begin{equation}
\begin{split}
c_N(z) &= c_0 + \sum_{n=1}^N \delta_n\\
&= \zeta_N(z) + \frac{1}{(N+1)^z}\frac{1-\overline{z}}{\abs{1-z}^2}\left( 1-x-\frac{y}{\tan(y\ln(\frac{N+2}{N+1}))} \right)
\end{split}
\end{equation}

Now, let us introduce:

\begin{equation}
\begin{split}
g_n(z) &= \frac{1}{(n+1)^z} \\
f_n(z) &= \frac{1-\overline{z}}{\abs{1-z}^2}\left( 1-x-\frac{y}{\tan(y\ln(\frac{N+2}{N+1}))} \right) \\
r_n(z) &= g_n(z)f_n(z)
\end{split}
\end{equation}

Because $ \sum J_{\delta_n} $ converges normally, this rearrangement is also true for the jacobian of $ c_n(z) $:

\begin{equation}
J_{c_n}((x, y)) = J_{\zeta_n}((x, y)) + J_{r_n}((x, y))
\end{equation}

$ \zeta_n(z) $ is a sum of holomorphic functions. Therefore, what we have to show is that $ J_{r_n}((x, y)) $ tends to satisfy the Cauchy-Riemann equation as $ n $ goes to infinity for any $ (x, y) \in \Omega $.

First, we compute the following quantities:

\begin{equation}
\begin{split}
\frac{\partial {f_n}_x}{\partial x} - \frac{\partial {f_n}_y}{\partial y} &\underset{n\rightarrow\infty}{=} \frac{1-x}{\abs{1-z}^2} + o(\frac{1}{n}) \\
\frac{\partial {f_n}_x}{\partial y} + \frac{\partial {f_n}_y}{\partial x} &\underset{n\rightarrow\infty}{=} \frac{1-x}{\abs{1-z}^2} + o(\frac{1}{n})
\end{split}
\end{equation}

Now, we can write:

\begin{equation}
\begin{split}
\frac{\partial {r_n}_x}{\partial x} &= \frac{\partial}{\partial x}\left( {f_n}_x{g_n}_x - {f_n}_y{g_n}_y \right) \\
&= \frac{\partial {f_n}_x}{\partial x}{g_n}_x + \frac{\partial {g_n}_x}{\partial x}{f_n}_x - \frac{\partial {f_n}_y}{\partial y}{g_n}_y - \frac{\partial {g_n}_y}{\partial y}{f_n}_y \\
\frac{\partial {r_n}_y}{\partial y} &= \frac{\partial}{\partial y}\left( {f_n}_x{g_n}_y + {f_n}_y{g_n}_x \right) \\
&= \frac{\partial {f_n}_x}{\partial y}{g_n}_y + \frac{\partial {g_n}_y}{\partial y}{f_n}_x + \frac{\partial {f_n}_y}{\partial y}{g_n}_x + \frac{\partial {g_n}_x}{\partial y}{f_n}_y
\end{split}
\end{equation}

Then, we take adventage from the fact that $ g_n $ is complex-differentiable:

\begin{equation}
\begin{split}
\frac{\partial {g_n}_x}{\partial x} - \frac{\partial {g_n}_y}{\partial y} &= 0 \\
\frac{\partial {g_n}_x}{\partial y} + \frac{\partial {g_n}_y}{\partial x} &= 0
\end{split}
\end{equation}

Therefore we find:

\begin{equation}
\abs{\frac{\partial {r_n}_x}{\partial x} - \frac{\partial {r_n}_y}{\partial y}} \leq \abs{\frac{\partial {f_n}_x}{\partial x} - \frac{\partial {f_n}_y}{\partial y}}\abs{{g_n}_x} + \abs{\frac{\partial {f_n}_x}{\partial y} + \frac{\partial {f_n}_y}{\partial x}}\abs{{g_n}_y}
\end{equation}

Exactly the same phenomenon happens for $ \abs{\frac{\partial {r_n}_x}{\partial y} + \frac{\partial {r_n}_x}{\partial y}} $, and we can write:

\begin{equation}
\begin{split}
\abs{\frac{\partial {r_n}_x}{\partial x} - \frac{\partial {r_n}_y}{\partial y}} &\underset{n\rightarrow\infty}{=} \frac{1}{(n+1)^x}\abs{\frac{1-x}{1-z}} + o(\frac{1}{(n+1)^x})\\
\abs{\frac{\partial {r_n}_x}{\partial y} + \frac{\partial {r_n}_x}{\partial y}} &\underset{n\rightarrow\infty}{=} \frac{1}{(n+1)^x}\abs{\frac{1-x}{1-z}} + o(\frac{1}{(n+1)^x})
\end{split}
\end{equation}

This is enough to conclude that $ c $ in holomorphic over $ \Omega $. In addition, it coincides with $ \zeta $ over $ \omega $ which is an open set, therefore we finally have:

\begin{equation}
c(z) = \zeta(z), z \in \Omega
\end{equation}

\subsection{Discussion one the radial convergence}

This summation procedure can remind the reader the link between the Riemann zêta function, and the Dirichlet êta function \cite{ARTICLE:2}:

\begin{equation}
\begin{split}
\eta(z) &= \sum_{n=1}^{\infty}\frac{(-1)^{n-1}}{n^z}, x>0 \\
\zeta(z) &= \frac{1}{1-2^{1-z}}\eta(z), z \neq 1, x>0
\end{split}
\end{equation}

Like the radial convergence, this definition is valid for $x>0$. However, there is a much more interesting observation to make here.

Indeed, applying Cesàro \cite{BOOK:2} summation on the original Dirichlet series makes it converge for $x>-1$. Applying it again will have the series to converge for $x>-2$. From here, one can show by inference that the Dirichlet series can be analytically extended to the whole complex plan by applying Cesàro summation infinitely many times.

It seems that we have exactly the same kind of property for the radial convergence applied on the Riemann series. This result is purely numerical and will not be proven here. But it appears that applying the radial convergence infinitely many times on the Riemann series makes it converge over the whole complex plan, private from the real axis.

An observation that is consistent with this idea is the following fact: For $x=0$, one can observe that the sequence $\left\{ c_n(z) \right\}_{z \in \mathbb{N}}$ converges to an asymptotic circle.

\section{Link with the non-trivial zeros of the Riemann zêta function}

In this section, $ \Omega $ is now the critical strip:

\begin{equation}
\Omega = \left\{x+iy, 0<x<1, y\in\mathbb{R}^* \right\}
\end{equation}

We are now going to change our point of view on $ c_n(z) $. Because it is the intersection of two straight lines in the complex plan, we can see $ c_n(z) $ as the solution of the following system:

\begin{equation}
\begin{bmatrix}
\Re c_n(z) \\
\Im c_n(z)
\end{bmatrix}
A_n = B_n
\label{azerty}
\end{equation}

With:

\begin{equation}
A_n =
\begin{bmatrix}
\Im \frac{e^{i\alpha}}{(n+1)^z} & -\Re \frac{e^{i\alpha}}{(n+1)^z} \\
\Im \frac{e^{i\alpha}}{(n+2)^z} & -\Re \frac{e^{i\alpha}}{(n+2)^z}
\end{bmatrix},
B_n =
\begin{bmatrix}
\Re \zeta_n(z) \Im \frac{e^{i\alpha}}{(n+1)^z} - \Im \zeta_n(z) \Re \frac{e^{i\alpha}}{(n+1)^z} \\
\Re \zeta_{n+1}(z) \Im \frac{e^{i\alpha}}{(n+2)^z} - \Im \zeta_{n+1}(z) \Re \frac{e^{i\alpha}}{(n+2)^z}
\end{bmatrix}
\end{equation}

We can now see explicitly that this system always has solutions:

\begin{equation}
det(A_n) = \frac{\sin(-y\ln(\frac{n+2}{n+1}))}{((n+1)(n+2))^{x}} \neq 0
\end{equation}

The complete expression of $ A_n^{-1} $ is:

\begin{equation}
A_n^{-1} =
\frac{((n+1)(n+2))^{x}}{\sin(-y\ln(\frac{n+2}{n+1}))}
\begin{bmatrix}
-\Re \frac{e^{i\alpha}}{(n+2)^z} & \Re \frac{e^{i\alpha}}{(n+1)^z} \\
-\Im \frac{e^{i\alpha}}{(n+2)^z} & \Im \frac{e^{i\alpha}}{(n+1)^z}
\end{bmatrix}
\end{equation}

We then rewrite $ B_n $ as:

\begin{equation}
B_n(z) =
\begin{bmatrix}
b_n(z) \\
b_{n+1}(z)
\end{bmatrix}
\end{equation}

With:

\begin{equation}
b_n = \frac{\abs{\zeta_n(z)}}{(n+1)^{x}}\sin\left(\arg\left(\frac{e^{i\alpha}}{(n+1)^z}\right) - \arg(\zeta_n(z))\right)
\end{equation}

For $ z \in \Omega $, it has already been proven that $ \abs{\zeta_n(z)} $ is unbounded.  Therefore, considering the \textbf{Equation \ref{azerty}}, the following assumption seems very reasonable, though it would deserve a explicit proof.

If $ z $ is a zero of $ \zeta $ over $\Omega$, then we have:

\begin{equation}
\abs{\sin\left(\arg\left(\frac{e^{i\alpha}}{(n+1)^z}\right) - \arg(\zeta_n(z))\right)} \underset{n\rightarrow\infty}{\rightarrow} 0
\end{equation}

Which with a bit of rearrangement can be formulated as the following theorem:

\begin{theorem}
\begin{equation}
z \in \Omega, \zeta(z)=0 \Rightarrow \underset{n\rightarrow\infty}{\lim} \frac{-y\ln(n+1) - \underset{sum}{\arg}(\zeta_n(z)) + \arctan(\frac{y}{1-x})}{\pi}  = \mathcal{U}_z, \mathcal{U}_z\in \mathbb{Z}
\label{reg_the}
\end{equation}
With $ \underset{sum}{\arg}(\zeta_n(z)) $ the cumulated argument of $ \zeta_n(z) $:
\begin{equation}
\underset{sum}{\arg}(\zeta_n(z)) = \sum_{j=1}^{n}\arg\left(\frac{\zeta_{j+1}(z)}{\zeta_j(z)}\right)
\end{equation}
\end{theorem}

Again, the complete proof of this theorem is not given here, but the numerical results are very consistent.

Let us call this integer $ \mathcal{U}_z $ when it is defined. For the known zeros of the Riemann zêta function, we can numerically compute this value. Here are the first 30 values for the 30 first non trivial zeros:

\begin{figure}[h]
\begin{center}
\begin{tabular}{|l|c|r|}
  \hline
  n & y(n) & U(1/2-i*y(n)) \\
  \hline
1	& 14.1347251417346937904572519835624766 & 8 \\
2	& 21.0220396387715549926284795938969162 & 14 \\
3	& 25.0108575801456887632137909925627734 & 18 \\
4	& 30.4248761258595132103118975305839571 & 24 \\
5	& 32.9350615877391896906623689640747418 & 28 \\
6	& 37.5861781588256712572177634807052984 & 32 \\
7	& 40.9187190121474951873981269146334247 & 38 \\
8	& 43.3270732809149995194961221654068456 & 40 \\
9	& 48.0051508811671597279424727494276636 & 46 \\
10	& 49.7738324776723021819167846785638367 & 48 \\
11	& 52.9703214777144606441472966088808216 & 52 \\
12	& 56.4462476970633948043677594767060321 & 56 \\
13	& 59.3470440026023530796536486749921759 & 60 \\
14	& 60.8317785246098098442599018245240815 & 64 \\
15	& 65.1125440480816066608750542531836072 & 68 \\
16	& 67.0798105294941737144788288965220700 & 72 \\
17	& 69.5464017111739792529268575265546586 & 76 \\
18	& 72.0671576744819075825221079698261175 & 78 \\
19	& 75.7046906990839331683269167620305404 & 84 \\
20	& 77.1448400688748053726826648563046925 & 88 \\
21	& 79.3373750202493679227635928771160578 & 90 \\
22	& 82.9103808540860301831648374947705599 & 94 \\
23	& 84.7354929805170501057353112068275569 & 96 \\
24	& 87.4252746131252294065316678509191351 & 100 \\
25	& 88.8091112076344654236823480795095125 & 104 \\
26	& 92.4918992705584842962597252418104965 & 108 \\
27	& 94.6513440405198869665979258152079645 & 110 \\
28	& 95.8706342282453097587410292192466718 & 114 \\
29	& 98.8311942181936922333244201386223539 & 118 \\
30	& 101.317851005731391228785447940292361 & 122  \\
  \hline
\end{tabular}
\caption{Values of $ \mathcal{U}_z $ for the first 30 non trivial zeros of zeta}
\label{Table}
\end{center}
\end{figure}

We can observe these values are all even, But there are no clear pattern in  their distribution despite this fact. The python code used to generate this table of values is on the GitHub. This sequence does not appear yet on: \href{https://oeis.org/}{https://oeis.org/} \cite{WEBSITE:2}.

Let us call:

\begin{equation}
\theta_n(z) = \arg\frac{\zeta_{n+1}(z)}{\zeta_n(z)}
\end{equation}

\begin{figure}[h]
\begin{center}
\includegraphics[scale=0.6]{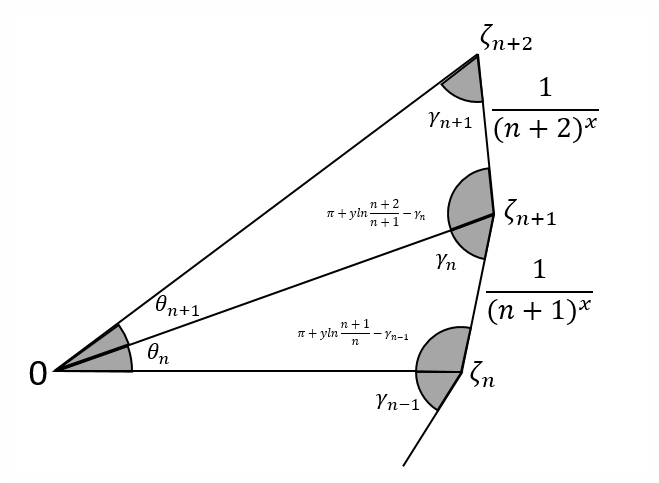}
\caption{Geometrical view of the situation for $ \theta_n(z) $}
\label{Goe3}
\end{center}
\end{figure}

In the figure \textbf{Figure \ref{Goe3}}, we also introduce $ \gamma_n(z) $ accordingly. One can show by induction that this $ \gamma_n(z) $ value is actually:

\begin{equation}
\gamma_n(z) = -y\ln(n+1) - \underset{sum}{\arg}(\zeta_n(z))
\end{equation}

Now, using geometrical considerations, it is possible to give the following induction formula for $ \sin(\theta_n(z)) $:

\begin{equation}
\sin(\theta_{n+1}(z)) = \frac{\sin(-y\ln(\frac{n+2}{n+1}) + \gamma_{n}(z))}{(n+2)^{x}\sqrt{\frac{\sin(-y\ln(\frac{n+1}{n}) + \gamma_{n-1}(z))^2}{(n+1)^{2x}\sin(\theta_n(z))^2} + \frac{1}{(n+2)^{2x}} + \frac{2\sin(-y\ln(\frac{n+1}{n}) + \gamma_{n-1}(z))}{(n+1)^{x}(n+2)^{x}\sin(\theta_n(z))}cos(-y\ln(\frac{n+2}{n+1}) + \gamma_{n}(z))}}
\end{equation}

This formula does not generalize nicely by induction. That is why the quantity $\underset{sum}{\arg}(\zeta_n(z))$ is not easy to compute.

However, from a geometrical point of view, because the Riemann series is notv bounded if $z$ is in the critical strip, one can remark that we must have $\gamma_n(z)$ to converge to $\alpha(z)$ modulo $2\pi$ as $n$ goes to infinity. This interpretation gives an explanation to why $\mathcal{U}_z$ is always even. One can remark it is in reality well defined all over the critical strip, though its convergence rate appears to be slower. From here, we can plot $\mathcal{U}_z$ for $x=\frac{1}{2}$ (\textbf{Figure \ref{Uz}}).

From this plot, we can clearly see that the zeros of the Riemann zêta function play a special role is the distribution of the values of $\mathcal{U}_z$.

\begin{figure}[h]
\begin{center}
\includegraphics[scale=0.3]{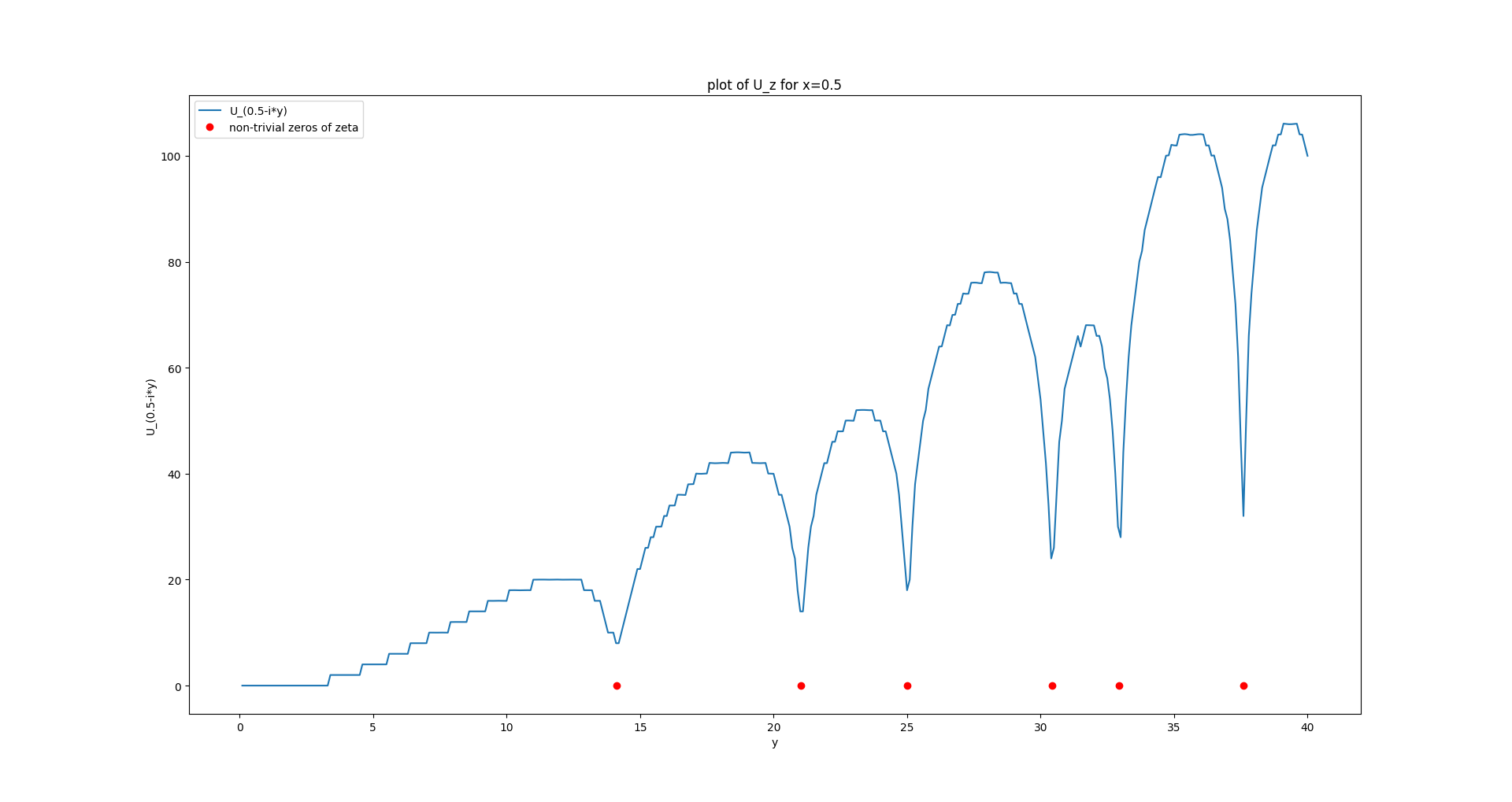}
\caption{Plot of $ \mathcal{U}_z $ for $z=\frac{1}{2}-y$}
\label{Uz}
\end{center}
\end{figure}

\bibliography{biblio} 
\bibliographystyle{ieeetr}

\end{document}